\documentclass{article}
\usepackage{amssymb}
\usepackage{amsfonts}
\usepackage{amsmath}
\usepackage{fullpage}
\usepackage{graphicx}
\usepackage{mathrsfs}
\usepackage{multirow}
\usepackage{setspace}
\usepackage[margin=1in]{geometry}

\newtheorem{theorem}{Theorem}

\newtheorem{corollary}[theorem]{Corollary}

\newtheorem{definition}[theorem]{Definition}

\newtheorem{proposition}[theorem]{Proposition}

\newenvironment{proof}[1][Proof]{\noindent{\sc#1} }{\nopagebreak\hfill \rule{0.5em}{0.5em}\\ }
\begin{document}

\title{A special family of Galton-Watson processes with explosions}
\author{Serik Sagitov and Alexey Lindo \\{\it Chalmers University of Technology and University of Gothenburg}
}
\maketitle

\begin{abstract}
The linear-fractional Galton-Watson processes is a well known case when many characteristics of a branching process can be computed explicitly. In this paper we extend the two-parameter  linear-fractional family to  a much richer  four-parameter family of  reproduction laws. The corresponding Galton-Watson processes also allow for explicit calculations, now with possibility for infinite mean, or even infinite number of offspring. We study the properties of this special family of branching processes, and show, in particular,  that in some explosive cases the time to explosion can be approximated by the Gumbel distribution. \let\thefootnote\relax\footnote{This research was supported by the Swedish Research Council grant
621-2010-5623.}
\end{abstract}

\renewcommand{\arraystretch}{1.5}

\section{Introduction}
Consider a Galton-Watson process $(Z_n)_{n\ge0}$ with $Z_0=1$ and the offspring number distribution 
$$p_k=P(Z_1=k),\quad k\ge0.$$
The properties of this branching process are studied in terms of  the probability generating function
$$f(s)=p_0+p_1s+p_2s^2+\ldots,
$$ 
where it is usual to assume that $f(1)=1$, however, in this paper we allow for $f(1)<1$ so that  a given particle may explode with probability $p_\infty=1-f(1)$.
 The  probability generating function  $f_n(s)=E(s^{Z_n})$ of the size of the  $n$-th generation is given by the $n$-fold iteration of $f(s)$
\[f_0(s)=s,\quad f_n(s)=f(f_{n-1}(s)),\quad n\ge1,\]
and therefore it is desirable to have a range of probability generating functions $f$ whose iterations can be computed explicitly.

The best known case of explicit calculations is the family of linear-fractional Galton-Watson processes with
\[f(s)=p_0+(1-p_0){ps\over1-(1-p)s},\quad s\in[0,(1-p)^{-1}),\]
representing the family of modified geometric distributions 
\[p_k=(1-p_0)(1-p)^{k-1}p,\quad k\ge1,\] 
fully characterized by just two parameters: $p_0\in[0,1)$ and $p\in(0,1]$. In Section \ref{Smod} for each $\theta\in[-1,1]$ we introduce a family $\mathcal G_\theta$ of 
functions with explicit iterations containing the linear-fractional family as a particular case. 
In Section \ref{Sam} we demonstrate that all $f\in \mathcal G_\theta$ are probability generating functions with $f(1)\le1$. A Galton-Watson processes with the reproduction law whose probability generating function belongs to $\mathcal  G_\theta$ will be  called  a theta-branching process.

The basic properties of the theta-branching processes are summarized in Section \ref{Sbr}, where it is shown that this family is wide enough to include the cases of infinite variance, infinite mean, and even non-regular branching processes with explosive particles. 

Recall that the basic classification of the Galton-Watson processes refers to the mean  offspring number $m=EZ_1$.
 Let $q\in[0,1]$ be the smallest non-negative root of the equation $f(x)=x$ and denote by
 \begin{align*}
   T_0&= \inf \{n: Z_{n} = 0 \}
  \end{align*}
the extinction time of the branching process. Then $q=P(T_0<\infty)$
  gives the probability of ultimate  extinction. For $m\le1$ and $p_1<1$, the extinction probability is $q=1$, while in the supercritical case $m>1$,  we have $q<1$. 
  
  If $f(1)<1$, then the Galton-Watson process is a Markov chain with two absorption states $\{0\}$ and  $\{\infty\}$. In this case the branching process either goes extinct at time $T_0$  or explodes at the time
 \begin{align*}
   T_1&= \inf \{n: Z_{n} = \infty \},
  \end{align*}
  with
  \[P(T_1\le n)=1-f_n(1),\quad P(T_1<\infty)=1-q,\]
where the latter equality is due to $f_n(1)\to q$. In Section \ref{Sexx}, using explicit formulas for $f_n(s)$ we compute the distribution of the absorption time 
$$T=T_0\wedge T_1.$$
Note that in the regular case, we have  $P(T_1=\infty)=1$ and therefore, $T\equiv T_0$.
Observe also that the case $f(1)<1$ has other, biologically more relevant interpretations. For example in the multitype setting, $T_1$ can be viewed as the time of the first mutation event, see \cite{SaS}.

Also in Section \ref{Sexx} we consider a situation when the explosion of a single particle has a small probability, so that $T_1$ takes large values in explosion scenarios. We show that in such a case the time to explosion can be asymptotically characterized with help of a Gumbel distribution.
In Section \ref{ququ} we study the $Q$-processes for the theta-branching processes extending the classical definition to the non-regular case. Our explicit calculations demonstrate that in the non-regular case the behavior of a branching process is more similar to that of the subcritical rather than supercritical regular case. Using these results on the $Q$-processes we derive the conditional limits of the theta-branching processes conditioned on non-absorption. 

A remarkable property of the linear-fractional Galton-Watson processes is that they can be embedded into the linear birth-death processes. In Section \ref{emb} we establish embeddability of theta-branching processes.

\section{Probability generating functions for theta-branching processes}\label{Smod}
Using an alternative parametrization for the linear-fractional probability generating functions, we obtain 
\begin{align}\label{lf}
{1\over 1-f(s)}={a\over 1-s}+c, \quad s\in[0,1),
\end{align}
where 
$$a={p\over 1-p_0},\quad c={1-p\over 1-p_0}.$$
This observation immediately implies that the $n$-fold iteration $f_n$ of the linear-fractional $f$ is also linear-fractional
\[{1\over 1-f_n(s)}={a^n\over 1-s}+c(1+a+\ldots+a^{n-1}).\]
%
The key idea of this paper is to  expand the family \eqref{lf} by
\begin{align}\label{AT}
(A-f(s))^{-\theta}&=a (A-s)^{-\theta}+c, \quad s\in[0,A),
\end{align}
 with the help of two extra parameters $(A,\theta)$ which are invariant under iterations.

%

\begin{definition}\label{d1}
Let $\theta\in(-1,0)\cup(0,1]$. We say that a probability generating function $f$ belongs to the family $\mathcal G_\theta$ if 
 \[
f(s)=
A-[a(A-s)^{-\theta}+c]^{-1/\theta},  \qquad 0\le s<A,
\]
where one of the following three options holds
\[
\begin{array}{cllll}
(i)  &  a\ge1,  &  c>0, &\theta\in(0,1],  & A=1, \\
(ii)  & a\in(0,1),&  c=(1-a) (1-q)^{-\theta}, &q\in[0,1), & A=1, \\
(iii)  & a\in(0,1),&  c=(1-a) (A-q)^{-\theta}, &q\in[0,1], & A>1.
\end{array}
\]
\end{definition}
Definition \ref{d1} can be extended to the case $\theta=0$ by the following continuity argument: for $a\in(0,1)$
\begin{align*}
A-[a(A-s)^{-\theta}+(1-a) (A-q)^{-\theta}]^{-1/\theta}
\to A-(A-q)^{1-a}(A-s)^{a}, \quad \theta\to0.
\end{align*}

\begin{definition}\label{d2}
We say  a probability generating function $f$ belongs to  
\begin{itemize}
\item the family  $\mathcal G_0$ if for some $a\in(0,1)$,
 \[
f(s)=
A-(A-q)^{1-a}(A-s)^{a},   \qquad 0\le s<A,
\]
where either  $A=1$,  $q\in[0,1)$, or $A>1$,  $q\in[0,1]$,

\item the family $f\in\mathcal G_{-1}$ if for some $q\in[0,1]$ and $a\in(0,1)$,
$$f(s)=as+(1-a)q,   \qquad 0\le s<\infty.$$
\end{itemize}
\end{definition}

%
%

\begin{definition}
A Galton-Watson process with the reproduction law whose probability generating function $f\in\mathcal G_\theta$, $\theta\in[-1,1]$,  will be called a theta-branching process.

\end{definition}

%
%
It is straightforward to see, cf. Section \ref{Sbr}, that each of the families $\mathcal G_\theta$ 
is invariant under iterations: if $f\in\mathcal G_\theta$, then $f_n\in\mathcal G_\theta$ for all $n\ge1$. The fact, that the functional families in Definitions \ref{d1} and  \ref{d2} are indeed consist of probability generating functions with $f(1)\le1$, is verified in Section \ref{Sam}.

Parts of the $\mathcal G_\theta$ families were mentioned earlier in the literature as examples of probability generating functions with explicit iterations.  
Clearly, $\mathcal G_1\cup \mathcal G_{-1}$ is the family of linear-fractional probability generating functions.
Examples in \cite{Z} leads to the case $A=1$ and $\theta\in[0,1)$, which was later given among other examples in Ch. 1.8 of \cite{Sevastianov1971}.  The case $A=1$ and $\theta\in(0,1)$ was later studied in \cite{To}. 
A special pdf with $\theta=-1/2$, 
\[f(s)=1-(a\sqrt{1-s}+1-a)^2,\quad a\in(0,1),\]
can be found in  \cite{Ha} on page 112, as an example of non-regular Galton-Watson processes. 

Notice that there is a version of linear-fractional Galton-Watson processes with countably many types of particles, see \cite{S}. It is an open problem to expand the  
theta-branching processes with $\theta\in(-1,1)$ to the multitype setting.

\section{Monotonicity properties}\label{Sam}
It is straightforward to see that each $f\in\mathcal G_0$ is a probability generating function with
\begin{align*}
f'(s)&=(A-q)^{1-a} a(A-s)^{a-1},\\ 
f^{(n)}(s)&=(A-q)^{1-a} a(1-a)\ldots(n-1-a)(A-s)^{a-n},\ n\ge2,
\end{align*}
and 
\begin{align*}
 p_0&=A-(A-q)^{1-a}A^{a},\\
  p_1&=(A-q)^{1-a} aA^{a-1},\\
 p_n&=p_{n-1}{n-a-1\over nA},\quad n\ge2.
\end{align*}
Therefore, $(p_n)_{n\ge1}$ are monotonely decreasing with  
\[p_n=aA^{a}(A-q)^{1-a} A^{-n}\prod_{k=2}^n\Big(1-{1+a\over k}\Big),\quad n\ge2,\]
so that $p_n\sim {\rm const}\cdot A^{-n}n^{-1-a}$ as $n\to\infty$.

\begin{proposition}
Let $\theta \in(-1,0)\cup(0,1)$ and $f\in\mathcal G_\theta$. Then $f$ is a probability generating function with $f(1)\le1$ such that 
\begin{align*}
 p_0&=A-(aA^{-\theta}+c)^{-1/\theta},\\
 p_1&=a(a+cA^{\theta })^{-1-1/\theta },
\end{align*}
and for $n\ge2$,
\[p_n={aA^{-n+1}\over (a+cA^\theta)^{\frac{1+\theta}{\theta}}n!}\cdot \sum_{i=1}^{n-1}\Big({cA^\theta\over a+cA^\theta}\Big)^iB_{i,n},\]
where all $B_{i,n}=B_{i,n}(\theta)$ are non-negative and, for $n\ge2$, satisfy the recursion
\begin{align*}
B_{i,n}=(n-2-i\theta)B_{i,n-1}+(1+i\theta)B_{i-1,n-1}, \ i=1,\ldots, n-1,
\end{align*}
 with $B_{0,n}=B_{n,n}=0$ for $n\ge1$, and $B_{1,2}=1+\theta$.
\end{proposition}

\begin{proof}
In terms of 
\[\phi(s):={A-f(s)\over A-s}=[a+c(A-s)^{\theta }]^{-1/\theta},\quad \phi'(s)=c(A-s)^{\theta -1}\phi(s)^{1+\theta},\]
we have 
\begin{align*}
 f'(s)&=a\phi(s)^{1+\theta},\\
f''(s)&=(1+\theta )ac(A-s)^{\theta -1}\phi(s)^{1+2\theta},\\
f'''(s)&=(1+\theta )(1-\theta )ac(A-s)^{\theta -2}\phi(s)^{1+2\theta}+(1+\theta )(1+2\theta )ac^2(A-s)^{2\theta -2}\phi(s)^{1+3\theta}.
\end{align*}
and more generally,
\[f^{(n)}(s)=\sum_{i=1}^{n-1}B_{i,n}ac^i(A-s)^{i\theta -n+1}\phi(s)^{1+(i+1)\theta},\quad n\ge2,\]
where $B_{i,n}$ are defined in the statement.
To finish the proof it remains to apply the  equality $p_n=f^{(n)}(0)/n!$. 
\end{proof}

In the linear-fractional case we have $p_k\ge p_{k+1}$ for all $k\ge1$. The next extension of this monotonicity property was first established in \cite{To}. 
\begin{corollary}
Let $\theta \in(0,1)$ and $f\in\mathcal G_\theta$ with  $A=1$. Then $p_k\ge p_{k+1}$ for all $k\ge1$. 
\end{corollary}
\begin{proof} Put 
 \[g(s)=(s-1)f(s)=-p_0+\sum_{k=1}^\infty (p_{k-1}-p_k)s^k\]
 From
 \begin{align*}
 g(s)&=s-1+(1-s)^2[a+c(1-s)^{\theta}]^{-1/\theta },\\
 g'(s)&=1+c(1-s)^{\theta+1}[a+c(1-s)^{\theta}]^{-1-1/\theta }-2(1-s)[a+c(1-s)^{\theta}]^{-1/\theta }\\
&=c(1-f(s))^{1+\theta}+2f(s)-1,
\\
 g''(s)&=(2-c(1+\theta)(1-f(s))^{\theta})f'(s),
\end{align*}
we see that $ g''(s)\ge0$, since
\[G(s):=2-c(1+\theta)(1-f(s))^{\theta}\ge2-c(1+\theta)(1-p_0)^{\theta}=2-{c(1+\theta)\over a+c}>0.\]
Furthermore,
 \begin{align*}
  G'(s)&=c\theta(1+\theta)(1-f(s))^{\theta-1}f'(s)
\end{align*}
is absolutely monotone (as a product of two absolutely monotone functions), implying that $ g''(s)$ is absolutely monotone, so that 
$$k(k-1)(p_{k-1}-p_k)\ge0,\quad k\ge2.$$

\end{proof}

\section{Basic properties of $f\in\mathcal G_\theta$}\label{Sbr}

In this section we distinguish among nine cases inside the collection of families $\{\mathcal G_\theta\}_{-1\le\theta\le1}$ and summarize the following basic fomulas: $f_n(s)$, $f(1)$, $f'(1)$,  $f''(1)$. In all cases, except Case 1, we have $a=f'(q)$. The following definition, cf \cite{KRS}, explains an intimate relationship between the Cases 3-5 with $A=1$ and  the Cases 7-9 with $A>1$.

\begin{definition}\label{deaf}
 Let $A>1$ and a probability generating function $f$ be such that $f(A)\le A$. We call  
 $$\hat f(s):={f(sA)\over A}=\sum_{k=0}^\infty p_kA^{k-1}s^k$$ 
 the dual generating function for $f$ and denote $\hat q=qA^{-1}$, so that $\hat f(\hat q)=\hat q$.  Clearly, $\hat f'(\hat q)=f'(q)$. 
 \end{definition}



{\bf Case 1}: $\theta\in(0,1]$, $a\in(1,\infty)$,
\[f_n(s)=1-[a^n(1-s)^{-\theta }+(a^n-1)d]^{-1/\theta },\quad  d\in(0,\infty).\]
The  corresponding theta-branching process is subcritical with $ m=a ^{-1/\theta }$.
If $\theta\in(0,1)$, then $ f''(1)=\infty$ and for $\theta = 1$ we have $f''(1) = 2(a - 1)a^{-2}d$.\\

{\bf Case 2}: $\theta\in(0,1]$, $a=1$,
\[f_n(s)=1-[(1-s)^{-\theta }+nc]^{-1/\theta },\quad c\in(0,\infty).\]
The corresponding theta-branching process is critical with either finite or infinite variance.
If $\theta \in (0, 1)$, then $f''(1)=\infty$ and for $\theta = 1$ we have $f''(1) = 2c$.
This is the only critical case in the whole family of theta-branching process.\\

{\bf Case 3}: $\theta\in(0,1]$,  $a\in(0,1)$,
\[f_n(s)=1-\big[a ^n(1-s)^{-\theta }+(1-a ^n)(1-q)^{-\theta }\big]^{-1/\theta },\quad q\in[0,1).\]
The corresponding theta-branching process is supercritical with $ m=a ^{-1/\theta }$.
If $\theta\in(0,1)$, then $ f''(1)=\infty$, and for $\theta = 1$ we have $f''(1)=2a^{-2}(1-a)(1 - q)^{-1}$.\\

{\bf Case 4}: $\theta=0$,  $a\in(0,1)$,
\[f_n(s)=1-(1-q)^{1-a^n}(1-s)^{a^n},\quad q\in[0,1).\]
The theta-branching process is regular supercritical with infinite mean.\\

{\bf Case 5}: $\theta\in(-1,0)$, $a\in(0,1)$,
\[f_n(s)=1-\big[a ^n(1-s)^{|\theta|}+(1-a ^n)(1-q)^{|\theta|}\big]^{1/{|\theta|}},\quad q\in[0,1).\]
The theta-branching process is non-regular with a positive
$$1-f(1)=(1-a)^{1/|\theta|}(1-q)$$
and infinite $f'(1)$.\\

{\bf Case 6}: $\theta =-1$, $a\in(0,1)$,
$$f_n(s)=a^ns+(1-a^n)q, \quad q\in[0,1].$$
If $q=1$, then  the theta-branching process becomes a pure death process with mean $m=a$ and $f''(1)=0$.
If $q<1$, then the theta-branching process is non-regular with a positive
$$1-f(1)=(1-a)(1-q),$$
$f'(1) = a$ and $f''(1) = 0$.\\

{\bf Case 7}: $\theta\in(0,1]$, $a\in(0,1)$,  $A>1$,
\[f_n(s)=A-[a^n(A-s)^{-\theta}+(1-a^n)(A-q)^{-\theta}]^{-1/\theta}, \quad q\in[0,1].\]
If $q=1$, then the  corresponding theta-branching process is subcritical with the offspring mean $m=a$ and
 $$f''(1)=(1+\theta)a(1-a)(A-1)^{-1}.$$
If $q\in[0,1)$, the theta-branching process is non-regular with a positive
$$1-f(1)=(A - 1)([a + (1 - a)(A - q)^{-\theta}(A - 1)^{\theta}]^{-1/\theta} - 1),$$
and
$$f'(1)=a[a+ (1 - a)(A - q)^{-\theta}(A - 1)^{\theta}]^{-1/\theta - 1},$$
$$f''(1)=(1+\theta)a(1 - a)(A - q)^{-\theta}(A-1)^{\theta-1}[a+(1 - a)(A - q)^{-\theta}(A - 1)^{\theta}]^{-1/\theta - 2}.$$
We have $f(A)=A$, and the dual generating function has the form of  the Case 3:
$$\hat f(s)=1-[a(1-s)^{-\theta  }+(1-a)(1-\hat q)^{-\theta  }]^{-1/\theta  }.$$

{\bf Case 8}: $\theta=0$,  $a\in(0,1)$, $A>1$, 
\[f_n(s)=A-(A-q)^{1-a^n}(A-s)^{a^n}, \quad q\in[0,1].\]
If $q=1$, the theta-branching process is subcritical with the offspring mean $m=a$ and
$$f''(1)=a(1-a)(A-1)^{-1}.$$
If  $q\in[0,1)$, the  theta-branching process is non-regular with a positive
$$1-f(1)=(A-q)^{1-a}(A-1)^{a}-(A-1),$$
and
$$f'(1) = a(A - q)^{1 - a}(A - 1)^{a - 1},$$
$$f''(1) = a(1 - a)(A - q)^{1 - a}(A - 1)^{a - 2}.$$
We have
$f(A)=A$,
and the dual generating function  belongs to the Case 4:
$$\hat f(s)=1-(1-\hat q)^{1-a}(1-s)^a.$$

{\bf Case 9}: $\theta\in(-1,0)$,  $a\in(0,1)$, $A>1$,
\[f_n(s)=A-\big[a ^n(A-s)^{|\theta|}+(1-a ^n)(A-q)^{|\theta|}\big]^{1/{|\theta|}}, \quad q\in[0,1].\]
If $q=1$, then the theta-branching process is subcritical with the offspring mean $m=a$ and
 $$f''(1)=(1 - |\theta|)a(1-a)(A-1)^{-1}.$$
If  $q\in[0,1)$, the theta-branching process is non-regular with a positive
$$1-f(1)=[a(A-1)^{|\theta|}+(1-a)(A-q)^{|\theta|}]^{1/|\theta|}-(A-1),$$
and 
$$f'(1) = a[a + (1 - a)(A - q)^{|\theta|}(A - 1)^{-|\theta|}]^{1/|\theta| - 1} \in(0,1),$$
$$f''(1)=(1 - |\theta|)a(1 - a)(A - q)^{|\theta|}(A - 1)^{-|\theta| - 1}[a + (1 - a)(A - q)^{|\theta|}(A - s)^{-|\theta|}]^{1/|\theta| - 2}.$$
With 
$$f(A)=A-(1-a)^{1/{|\theta|}}(A-q)\in(q,A),$$
the dual generating function  takes the form of  the Case 5:
$$\hat f(s)=1-[a(1-s)^{|\theta|  }+(1-a)(1-\hat q)^{|\theta|  }]^{1/|\theta|  }.$$ 

%
%


\section{Extinction and explosion times}\label{Sexx}
Recall that $T=T_0\wedge T_1$, and in the regular case $T=T_0$.
In the non-regular case, when $f(1)<1$, from
\begin{align*}
  P(n < T_{0} < \infty) &= q - f_{n}(0),\\
  P(n < T_{1} < \infty) &=  f_{n}(1)-q,
\end{align*}
we obtain
\begin{align*}
 P(n<T< \infty)&=f_n(1)-f_n(0).
\end{align*}
For our special family of branching processes we compute explicitly the distribution functions of the times $T_0, T_1, T$. \\

{\bf Cases 1-4.} In these regular cases we are interested only in the extinction time:

\[
 P(n < T_{0} < \infty)=\left\{
\begin{array}{ll}
 a^{-n/\theta}[1+d-da^{-n}]^{-1/\theta},  & \mbox{Case 1},  \\
(1+cn)^{-1/\theta },  &   \mbox{Case 2},    \\
(1 - q) \big( [1 - a^{n}(1 - (1 - q)^{\theta})]^{-1/\theta} - 1 \big),  &  \mbox{Case 3},    \\
  (1 - q) [(1 - q)^{-a^{n}} - 1],   &  \mbox{Case 4}.
\end{array}
\right.
\]


{\bf Cases 5, 7, 9.} In these cases
\begin{align*}
  P(n < T_{0} < \infty) &= (A - q)([1 - a^{n}(1 -(A -q)^{\theta}A^{-\theta})]^{-1/\theta} - 1), \\
 P(n < T_{1} < \infty) &= (A - q) \big( 1 - [1 - a^{n}(1 - (A - q)^{\theta}(A - 1)^{-\theta})]^{-1/\theta} \big), \\
 P(n < T< \infty) &= (A - q) \Big\{ [1 - a^{n}(1 -(A -q)^{\theta}A^{-\theta})]^{-1/\theta} - [1 - a^{n}(1 - (A - q)^{\theta}(A - 1)^{-\theta})]^{-1/\theta}  \Big\}.
\end{align*}

{\bf Case 6.} In this trivial case
\begin{align*}
P(n < T_{0} < \infty) &= a^{n}q, \quad
P(n < T_{1} < \infty) = a^{n}(1 - q),\quad
P(n < T< \infty)  = a^{n}.
\end{align*}
and for $q\in(0,1)$,
\begin{align*}
E(T_{0}|T_{0} < \infty) &=
E(T_{1}|T_{1} < \infty) =E(T)= \frac{1}{1 - a}.
\end{align*}

{\bf Case 8.} In this case
\begin{align*}
  P(n < T_{0} < \infty) &= (A - q)[(A - q)^{-a^{n}}A^{a^{n}} - 1], \\
  P(n < T_{1} < \infty) &= (A - q)[1 - (A -q)^{-a^{n}}(A - 1)^{a^{n}}], \\
 P(n < T< \infty) &= (A - q)^{1 - a^{n}}[A^{a^{n}} - (A - 1)^{a^{n}}].
 \end{align*}


\begin{theorem} Consider a theta-branching process with $\theta\in(-1,0]$ and $A\ge 1$.
Let $\theta\to 0$ and $A\to1$ in such a way that 
\[|\theta|\cdot \log{1\over A-1}\to r, \quad r\in[0,\infty].\]
Then for any fixed $a\in(0,1)$, $q\in[0,1)$, and $y\in(-\infty,\infty)$,
 \begin{equation*}
    \lim_{\epsilon \to 0} P( T_1 - \log_{a}\epsilon  \le y | T_1<\infty) = e^{-wa^{y}},
  \end{equation*}
where
  \[
\epsilon=\left\{
\begin{array}{ll}
 |\theta|, & r\in(0,\infty],   \\
( \log{1\over A-1})^{-1},  &  r=0,
\end{array}
\right.
\qquad 
w=\left\{
\begin{array}{ll}
 1, & r\in\{0\}\cup\{\infty\},   \\
1-e^{-r},  &  r\in(0,\infty).
\end{array}
\right.
\]
The limit is a Gumbel distribution with mean  ${\log w-\gamma\over\log{a}}$,
where $\gamma$ is the Euler-Mascheroni constant.

\end{theorem}
\begin{proof}  In view of
\begin{align*}
P(T_1\le n|T_1<\infty)={A-q\over 1-q}\big[1-a ^n(1-(A-1)^{|\theta|}(A-q)^{-|\theta|})\big]^{1/{|\theta|}}-{A-1\over 1-q},
\end{align*}
it suffices to verify that 
\begin{align*}
\big[1-\epsilon a ^y (1-(A-1)^{|\theta|})\big]^{1/{|\theta|}}\to e^{-w a^{y}}.
\end{align*}
Indeed, if $r=\infty$, then $(A-1)^{|\theta|}\to0$, and 
\begin{align*}
\big[1-|\theta|a ^y (1-(A-1)^{|\theta|})\big]^{1/{|\theta|}}\to e^{-a^{y}}.
\end{align*}
If $r\in(0,\infty)$, then $(A-1)^{|\theta|}\to e^{-r}$, and 
\begin{align*}
\big[1-|\theta|a ^y(1-(A-1)^{|\theta|})\big]^{1/{|\theta|}}\to e^{-a^{y}(1-e^{-r})}.
\end{align*}
Finally, if $r=0$, then 
$$1-(A-1)^{|\theta|}\sim |\theta|/\epsilon,$$
and therefore
\begin{align*}
\big[1-\epsilon a ^y(1-(A-1)^{|\theta|})\big]^{1/{|\theta|}}\to e^{-a^{y}}.
\end{align*}

\end{proof}
\begin{corollary}
If $A=1$ and $\theta\in(-1,0)$, then for any fixed $a\in(0,1)$ and $q\in[0,1)$,
  \begin{equation*}
    \lim_{\theta \to 0} P( T_1 - \log_{a} |\theta| \le y | T_1<\infty) = e^{-a^{y }},\quad y\in(-\infty,\infty),
  \end{equation*}
If $\theta=0$ and  $A=1+e^{-1/\epsilon}$, $\epsilon>0$, then for any fixed $a\in(0,1)$ and $q\in[0,1)$,
 \begin{equation*}
    \lim_{\epsilon\to 0} P( T_1 - \log_a\epsilon  \le y | T_1<\infty) = e^{-a^{y }},\quad y\in(-\infty,\infty).
  \end{equation*}
\end{corollary}

\section{The $Q$-process}\label{ququ}

As explained in Ch I.14,  \cite{AN}, for a regular Galton-Watson process with transition probabilities $P_n(i,j)$, one can define another Markov chain  with transition probabilities 
\[Q_{n}(i,j):={jq^{j-i}P_n(i,j)\over \gamma^ni}, \quad i\ge1,\ ,j\ge1,\]
where  $ \gamma=f'(q)$.
The new chain is called the $Q$-process, and from
\[\sum_{j\ge1}Q_{n}(i,j)s^j={s\over  \gamma^niq^i}{d\over ds}(f^i_n(sq)")=s\cdot {f_n'(sq)\over f_n'(q)}\cdot \Big({f_n(sq)\over q}\Big)^{i-1}\]
we see that the $Q$-process is a Galton-Watson process with the dual reproduction ${f(sq)\over q}$ and an eternal particle generating a random number $\kappa$ of ordinary particles with $E(s^\kappa)={f'(sq)\over f'(q)}$, see \cite{KRS}. The $Q$-process in the regular case is interpreted in \cite{AN} as the original branching process "conditioned on not being extinct in the distant future and on being extinct in the even more distant future".

Exactly the same definition of  the $Q$-process makes sense in the non-regular case, only now the last interpretation should be based on the absorption time $T$ rather than on the extinction time $T_0$. Indeed, writing $P_j(\cdot)=P(\cdot|Z_0=j)$ we get for $j\ge1$,
\begin{align*}
P_j(T>n)&=f^j_n(1)-f^j_n(0),
\end{align*}
and therefore,
\begin{align*}
P_i(Z_1=j_1&,\ldots,Z_{n}=j_n |T>n+k)
=P_i(Z_1=j_1,\ldots,Z_n=j_n ){f^{j_n}_k(1)-f^{j_n}_k(0) \over f^i_{n+k}(1)-f^i_{n+k}(0)}.
\end{align*}
In the non-regular case, as $k\to\infty$ we have $f_k(0)\to q$ and $f_k(1)\to q$. Thus, repeating the key argument of Ch I.14,  \cite{AN} for the derivation of the $Q$-process,
\begin{align*}
P_i(Z_1=j_1&,\ldots,Z_{n}=j_{n}  |T>n+k)\to P_i(Z_1=j_1,\ldots,Z_n=j_n){j_nq^{j_n} \over  \gamma^niq^i},
\end{align*}
we arrive in the limit to a Markov chain with the transition probabilities $Q_n(i,j)$.

By  Theorem 3 from Ch. I.11 in  \cite{AN},
\[\gamma^{-n}P_n(i,j)\to iq^{i-1}\nu_j,\quad i,j\ge1,\]
where $Q(s)=\sum _{j\ge1 }\nu_js^j$ satisfies
\[Q(f(s))= \gamma Q(s),\quad Q(q)=0.\]
In the critical case as well as in the subcritical case with $\sum_{k=2}^\infty p_kk\log k=\infty$ the solution is trivial: $Q(s)\equiv0$. Otherwise, $Q(s)$ is uniquely defined by the above equation with an extra condition $Q'(q)=1$, so that the $Q$-process has a stationary distribution given by
\[Q_{n}(i,j)\to jq^{j-1}\nu_j, \]
with
$$\sum _{j\ge1 } jq^{j-1}\nu_js^j=sQ'(sq).$$

These facts concerning $Q(s)$ remain valid even in the non-regular case. It is easy see from \eqref{AT} that for our family with $ \theta\ne0$ and $A>q$, the generating function
\[Q(s)=
(A-s)^{-\theta}-(A-q)^{-\theta},\]
is determined by parameters $(\theta,A)$ and is independent of $a=\gamma$. Similarly, for $\theta=0$ we have
\[Q(s)=\log{A-s\over A-q}.\]
This leaves us with two cases when $A=q=1$. In the critical Case 2 the answer is trivial: $Q(s)\equiv0$. In the subcritical Case 1, we have $\gamma=a^{-1/\theta}$ and
$$(1-f(s))^{-\theta}+d=\gamma^{-\theta}[(1-s)^{-\theta}+d],$$
which yields
$$Q(s)=[(1-s)^{-\theta}+d]^{-1/\theta}.$$
From these calculations it follows, in particular, that for our family of branching processes, in all subcritical cases, the classical $x\log x$ moment condition holds:
\begin{equation*}\label{xlx}
 \sum_{k=2}^\infty p_kk\log k<\infty.
\end{equation*}

Using these explicit formulas for $Q(s)$ we can easily find the conditional probability distributions 
 \[\lim_{n\to\infty}P(Z_n=j|T>n)= b_j,\quad j\ge1.\]
For all cases, except the critical Case 2, we have
$$\sum_{j\ge1}b_js^j=1-{Q(sq)\over Q(0)}.$$
Turning to the Case 2,  recall that for any critical Galton-Watson process, there exists a limit probability distribution
 \[ \lim_{n \to \infty} P(Z_{n} = j | T_{0} = n + 1) = w_{j}, \quad j \geq 1,\]
such that
$$\sum_{j \geq 1} w_{j} s^{j} = \lim_{n \to \infty} \frac{f_{n}(sp_{0}) - f_{n}(0)}{f_{n}(p_{0}) - f_{n}(0)}.$$
Since
\begin{align*}
  f_{n}(sp_{0}) &= 1 - [(1 - s(1 - [1 + c]^{-1/\theta}))^{-\theta} + nc]^{-1/\theta},
\end{align*}
we obtain
$$\sum_{j \geq 1} w_{j} s^{j}  ={[1 - s(1 - [1 + c]^{-1/\theta})]^{-\theta} - 1\over c}.$$

\section{Embedding into continuous time branching processes}\label{emb}

Recall that a Galton-Watson processes with generating functions $f_n$ is called {\it embeddable}, if there is  a semigroup of probability generating functions 
\begin{equation}\label{semi}
 F_{t+u}(s)=F_t(F_u(s)),\quad t\in[0,\infty), u\in[0,\infty),
\end{equation}
 such that $f_n(s)=F_n(s)$, $n=1,2,\ldots$. Although not every Galton-Watson process is embeddable, see Ch. III.6 in \cite{AN}, in this section we demonstrate that all theta-branching processes are embeddable.
 
Behind each semigroup \eqref{semi} there is a continuous time Markov branching process with particles having exponential life lengths with parameter, say, $\lambda$. Each particle at the moment of death is replaced by a random number of new particles having a probability generating function 
\begin{equation*}
 h(s)=h_0+h_2s^2+h_3s^3+\ldots.
\end{equation*}
For such a continuous time branching process $(Z_t)_{t\in[0,\infty)}$ the probability generating function $F_t(s)=Es^{Z_t}$ satisfies
\begin{equation}\label{I}
 \int_s^{F_t(s)}{dx\over h(x)-x}=\lambda t
\end{equation}
(see \cite{Ss} for a recent account of continuous time Markov branching processes). Our task for this section  is for each $f\in\mathcal G_\theta$ to find a pair $(h,\lambda)$ such that  $f(s)=F_1(s)$.
We will denote by $\mu=\sum_{k=2}^\infty kh_k$ the corresponding offspring mean number and by $q$ the minimal nonnegative root of the equation $h(s)=s$ which gives the extinction probability of the continuous time branching process.\\

{\bf Cases 1-3.} For a pair $ \theta\in(0,1]$ and  $\mu\in(0, 1+\theta^{-1}]$, put
\[h(s)=1-\mu(1-s)+{\mu\over1+\theta}(1-s)^{1+\theta}.\]
Taking successive derivatives of $h$ it easy to see that it is a probability generating function with  $h'(0)=0$. Next we show that using this $h$ as the offspring probability generating function for the continuous time branching process we can recover $f(s)$ for the theta-branching processes as $F_1(s)$  by choosing $\mu$ and $\lambda$ adapted to Cases 1- 3.

Case 1. For a given pair $a\in(0,1)$ and $d\in(0,\infty)$, put
\[\mu ={(1 + \theta)d \over (1 + \theta)d + 1},\qquad \lambda = [ (1 + \theta^{-1})d + \theta^{-1}]\ln a.\]
In this subcritical case,  applying \eqref{I} we obtain for $s \in [0, 1)$
\[ \lambda t=\int_s^{F_t(s)}{dx\over (1-\mu)(1-x)+{\mu\over1+\theta}(1-x)^{1+\theta}}=\int_s^{F_t(s)}{d\log{1\over1-x}\over 1-\mu+{\mu\over1+\theta} e^{\theta\log(1-x)}},\]
yielding the desired formula
\[ F_{t}(s)= 1 - \Big( a^{t}(1 - s)^{-\theta} + (a^{t} - 1)d \Big)^{-1/\theta}.\]

Case 2. For a given $c\in(0,\infty)$, put $\mu = 1$ and $\lambda=(1 + \theta^{-1})c$. Then by \eqref{I}, we get
\[ F_{t}(s) = 1 - \Big( (1 - s)^{-\theta} + ct \Big)^{-1/\theta}. \]

Case 3. If $\mu>1$, then $q=1-({(\mu-1)(1+\theta)\over \mu})^{1/\theta}$ and the proposed $h$ can be rewritten as
\begin{equation*}
 h(s) = s + \frac{(1 - s)^{1 + \theta} - (1 - q)^{\theta}(1 - s)}{1 + \theta - (1 - q)^{\theta}}.
 \end{equation*}
For a given pair $a\in(0,1)$ and $q\in[0,1)$ choosing
\[\lambda = [ (1 + \theta^{-1})(1 - q)^{-\theta} -  \theta^{-1}]\ln a^{-1}\]
and applying \eqref{I}, we obtain
\[ F_{t}(s) = 1 - [a^{t}(1 - s)^{-\theta} + (1 - a^{t})(1 - q)^{-\theta}]^{-1/\theta}. \]
It is easy to see that $f(s)=F_1(s)$ covers the whole subfamily $\mathcal{G}_{\theta}$ corresponding to the Cases 1-3.

Notice that if $\theta=1$, then $h(s)=1-{\mu\over2}+{\mu\over2}s^2$ generates the linear birth and death process with  $h''(1) = \mu$.  If $\theta \in (0, 1)$, then $h''(1) = \infty$.\\

{\bf Case 4.}
Consider a supercritical reproduction law with infinite mean
\[ h(s) =  s +(1 - s)  {\ln(1 - s) - \ln(1 - q)\over 1 - \ln(1 - q) }. \]
For $h_{0} \in [0, 1)$ this can be rewritten as
\begin{equation*}
  h(s) = h_{0} + (1 - h_{0}) \sum_{k = 2}^{\infty} \frac{s^{k}}{k(k - 1)}.
\end{equation*}
In this form with $h_0=0$, the generating function $h$ appeared in~\cite{Lageraas2006} as the reproduction law of an immortal branching process.
Earlier in~\cite{Sevastianov1971}, this reproduction law was introduced as
\begin{equation*}
  h(s) = 1 - (1 - h_{0})(1-s)(1 - \ln(1 - s)).
\end{equation*}
To see that the theta-branching process in the Case 4  is embeddable into the Markov branching process with the above mentioned  reproduction law, use the first representation of $h$ and apply ~\eqref{I}. As a result we obtain for $s \ne q$, 
\begin{align*}
  {\lambda t \over 1-\ln(1 - q)} &= \int_{s}^{F_{t}(s)} \frac{dx}{(1 - x)(\ln(1 - x) - \ln(1 - q))} = \int_{s}^{F_{t}(s)} \frac{\ln(1 - x)}{\ln(1 - q) - \ln(1 - x)} \\
  &= \ln[\ln(1 - s) - \ln(1 - q)] - \ln[\ln(1 - F_{t}(s)) - \ln(1 - q)].
\end{align*}
Putting $\lambda = (1-\ln(1 - q)) \ln a^{-1}$, we derive
\begin{align*}
  F_{t}(s) = 1 - (1 - q)^{1 - a^{t}}(1 - s)^{a^{t}}.
\end{align*}

{\bf Cases 5, 7, 9.} In these three cases  the corresponding  $h$ and $\lambda$ are given by an extension of the formulas for the Case 3:
\[ h(s) = s + \frac{(A - s)^{1 + \theta} -(A - q)^{\theta}(A - s)}{(1 + \theta)A^{\theta} - (A - q)^{\theta}}, \quad \lambda = [ (1 + \theta^{-1})A^\theta(A - q)^{-\theta} -  \theta^{-1}]\ln a^{-1}. \]
Turning to Definition \ref{deaf} we see that this $h$ in the Case 7 is dual to the $h$ in the Case 3, and  in the Case 9 it is dual to that of the Case 5.\\

{\bf Case 6.}
In this trivial case the corresponding continuous time branching process is a simple death-explosion process with $h(s) = q$ and $\lambda = \ln a^{-1}$.\\

{\bf Case 8.}
Similarly to the Case 4 we find that the pair
\[ h(s) = s +(A  - s) { \ln (A - s)-\ln(A - q)\over 1 + \ln  A-\ln(A - q)}, \quad \lambda = (1 + \ln  A-\ln(A - q)) \ln a^{-1}, \]
lead to
\[F_t(s)=A-(A-q)^{1-a^t}(A-s)^{a^t}.\]
Observe that this $h$ is dual to that of the Case 4.

%

\end{document}